\documentclass[12pt,twoside,reqno]{amsart}

\allowdisplaybreaks
\newtheorem{thm}{Theorem}[section]
\newtheorem{lemma}{Lemma}[section]
\newtheorem{rem}{Remark}[section]

\theoremstyle{definition}

\theoremstyle{remark}

\newcommand{\R}{{\mathbb R}}
\newcommand{\N}{{\mathbb N}}

\newcommand{\noi}{\noindent}

\newcommand{\s}{\mathcal S}

\def \Om{\Omega}

\numberwithin{equation}{section}

\begin{document}
\title[On one problem of Saut-Temam]
{
	To one problem of Saut-Temam for the 3D Zakharov-Kuznetsov equation
}
\author{N. A. Larkin$^\dag$\;	\&
	M. V. Padilha}

\address
{
	Departamento de Matem\'atica, Universidade Estadual
	de Maring\'a, Av. Colombo 5790: Ag\^encia UEM, 87020-900, Maring\'a, PR, Brazil
}

\thanks
{
	$^\dag$
	Corresponding author
}
\bigskip

\email{ nlarkine@uem.br;nlarkine@yahoo.com.br; marcvfag@gmail.com }
\date{}

 \subjclass[2010]{35Q53,
35B40}
\keywords
{ZK equation, stabilization}

\begin{abstract}An initial-boundary value problem for the 3D Zakharov-Kuznetsov equation posed on an unbounded domain is considered.
Existence and uniqueness of a global regular solution as well as  exponential decay   of the $H^2$-norm for small initial data are proven.
\end{abstract}

\maketitle

\section{Introduction}\label{introduction}

We are concerned with the existence, uniqueness and exponential decay of the $H^2$-norm for global regular solutions   to an initial-boundary value problem (IBVP) for the 3D Zakharov-Kuznetsov (ZK) equation 
 \begin{equation}
u_t+(1 +u)u_x +u_{xxx}+u_{xyy}+u_{xzz}=0 \label{zk}
\end{equation}
which describes the propagation of nonlinear ionic-sonic waves in a plasma submitted to a magnetic field directed along the $x$ axis. This equation is a three-dimensional
analog of the well-known Korteweg-de Vries (KdV) equation
\begin{equation}\label{kdv}
u_t+uu_x+u_{xxx}=0.
\end{equation}

Equations \eqref{zk}, \eqref{kdv} are typical examples
of so-called
dispersive equations which attract considerable attention
of both pure and applied mathematicians in the past decades. The KdV
equation is probably most studied in this context.
The theory of the initial-value problem
(IVP henceforth)
for \eqref{kdv} is considerably advanced today
\cite{bona2,kato,ponce2,saut2,temam1}.

Recently, due to physics and numerics needs, publications on initial-boundary value
problems to \eqref{kdv} both  in bounded and unbounded domains for dispersive equations have appeared
\cite{bona1,larkin,lar2,wang}. In
particular, it has been discovered that the KdV equation posed on a
bounded interval possesses an implicit internal dissipation. This allowed
to prove the exponential decay rate of small solutions for
\eqref{kdv} posed on unbounded intervals without adding any
artificial damping term \cite{larkin}. Similar results were proved
for a wide class of dispersive equations of any odd order with one
space variable \cite{familark}.

However, \eqref{kdv} is a satisfactory approximation for real waves phenomena while the
equation is posed on the whole line ($x\in\mathbb{R}$); if
cutting-off domains are taken into account, \eqref{kdv} is no longer
expected to mirror an accurate rendition of reality. The correct
equation in this case (see, for instance, \cite{bona1})
should be written as
\begin{equation}\label{1.3}
u_t+ u_x+uu_x+u_{xxx}=0.
\end{equation}
Indeed, if $x\in\R,\ t>0$,
the linear traveling term $u_x$ in \eqref{1.3} can be easily scaled
out by a simple change of variables, but it can
not be safely ignored for problems posed  both on finite and semi-infinite intervals without
changes in the original domain.

Once bounded domains are considered
as a spatial region of waves propagation, their sizes appear to be
restricted by certain critical conditions.
 We recall, however, that
if the transport term $u_x$ is neglected, then \eqref{1.3} becomes \eqref{kdv}, and it is possible to prove the
exponential decay rate of small solutions for \eqref{kdv} posed
on any bounded interval.
More  results on control and stabilizability for the KdV equation can
be found in \cite{rosier1,rozan}.

Later, the interest on dispersive equations became to be extended for
multi-dimensional models such as Kadomtsev-Petviashvili (KP)
and ZK equations.
As far as the ZK equation is concerned,
 results  both on IVP and IBVP can be found in
\cite{faminski,faminski2,farah,pastor,pastor2,saut,ribaud,temam}.
The biggest part of these publications is devoted to study of well-posedness of the Cauchy problem and initial-boundary value problems for the 2D ZK equation
\cite{faminski, faminski2, farah,pastor, pastor2}. In the case of the 3D ZK equation, there are results on local well- posedness for the Cauchy problem \cite{saut,ribaud}; the existence of local strong solutions to an initial- boundary value problem posed on a bounded domain, \cite{wang}, as well as the existence of global weak solutions \cite{temam}.

 Our work has
been inspired by \cite{temam} where \eqref{zk} posed  on an unbounded domain was considered. A thorough analysis of these papers has revealed that an implicit dissipativity of the terms $u_{xyy}+u_{xzz}$ may help to establish a global well-posedness of initial-boundary value problems in classes of regular solutions. Yearlier this dissipativity has been used in order to prove exponential decay for the 2D ZK equation \cite{larh1,larkintronco}.

The main goal of our work is to prove the existence and uniqueness
of global-in-time regular solutions of \eqref{zk} posed  on
unbounded domains and the exponential decay rate of
these solutions for sufficiently small initial data. To cope with this problem, we exploited the strategy completely different from the standard schemes: first to prove the existence result and after that to study uniqueness and decay properties of solutions. In our case, we prove simultaneously existence of global regular solutions and their exponential decay.

The paper is outlined as follows. Section I is Introduction. Section 2 contains  formulation
of the problem and auxiliaries. In Section \ref{existence}, we prove the existence of global regular solutions and, simultaneously, exponential decay of the $H^2$-norm. In section \ref{uniqueness} uniqueness of a regular solutions and continuous dependence on initial data are proven.

\section{Problem, Preliminares and Main Result}

\hspace{0.5mm}Let $L>0$ be a finite number. Define $\Omega = \{(x,y,z) \in \R^3: x \in (0,L), y \in \R; z\in \R\},\;\;   \s = \R^2$.

Consider in $\Omega \times (0,t)$ the following initial- boundary value problem for the Zakharov-Kuznetsov equation:

	\begin{align}
		&  Lu = u_t + u_x + uu_x + \Delta u_x = 0;  \label{1problema}\\
		&	u(0,y,z,t) = u(L,y,z,t) =  u_x(L,y,z,t) = 0, \label{2problema}\\
		&	u(x,y,z,0) = u_0(x,y,z),\,\,\,\,  (x,y,z) \in \Omega. \label{3problema}
	\end{align}
	
	Hereafter subscripts $u_x,\ u_{xy},$ etc. denote the partial derivatives,
	as well as $\partial_x$ or $\partial_{xy}^2$ when it is convenient.
	Operators $\nabla$ and $\Delta$ are the gradient and Laplacian acting over $\Om.$
	By $(\cdot,\cdot)$ and $\|\cdot\|$ we denote the inner product and the norm in $L^2(\Om),$
	$\|\cdot\|_{H^k}$ stands for the norm in $L^2$-based Sobolev spaces, and $\|u\|_{[2]}^2=\|u_{xx}\|^2 + \|u_{yy} \|^2 + \|u_{zz}\|^2$.

		\begin{thm}\label{main2}Let $L \leq \frac{\pi}{2}$ and $u_0(x,y,z)$ satisfying the following conditions:
			 
			 $$u_0(0,x,y,z) = u_0(L,y,z,t) = u_{0x}(L,y,z,t) = 0,$$
			
			$$\| u_0\| ^2 + \|u_{0yy} \|^2 + \|u_{0yz}\|^2 + \| u_{0zz}\|^2 + J_0 < \infty,$$

			\begin{equation}\label{condition}\|u_0 \|^4 \leq \frac{\pi^2}{8K_1L^2}, \,\,\,\,\,\,\,\,\,\,\,\ J_0^2 \leq \frac{\pi^2}{200K_2 L^2},\end{equation} 		
			
			\noi where $K_1 = 2^{16} 3^3 (1+L)(\frac{2^3}{25}C_1 + 1)$, $C_1 = 2+\frac{2^{13}}{3} \|u_0\|^4$, $K_2 = 2^{19} 3^3(1+L)^6$ and $J_0 = ((1+x), | u_0 + u_{0x} + u_0 u_{0x} + \Delta u_{0x}|^2)$.
						
			\noi Then there exists a unique global regular solutions to (\ref{1problema} - \ref{3problema}):
			
			\begin{equation}u \in L^\infty (0,+\infty; H^2(\Omega)) \cap L^2(0,+\infty;H^3(\Omega)),\end{equation}
			\begin{equation}u_{t} \in L^\infty (0,+\infty; L^2(\Omega)) \cap L^2(0,\infty; H^1(\Omega)),\end{equation}
			\begin{equation}\Delta u_{x} \in L^\infty (0,+\infty; L^2(\Omega)) \cap L^2(0,\infty; H^1(\Omega))\end{equation}
			
			\noi such that
			
			\begin{equation}\|u \|_{H^2(\Omega)}^2(t) + \| u_t \|^2(t) \leq C(L, J_0)e^{-\frac{\chi}{2} t},\;\;\forall t>0,\end{equation}
			
			\noi where $\chi = \frac{\pi^2}{2L^2(1+L)}$.
			
		\end{thm}
		
		We will need the following results:

\begin{lemma}\label{lemma1}
	Let $u\in H^1(\Om)$ and $\gamma$ be the boundary of $\Om.$
	
	If $u|_{\gamma}=0,$ then
	\begin{equation}\label{lady1}
		\|u\|_{L^q(\Om)}\le 4^{\theta}\|\nabla u\|^{\theta}\|u\|^{1-\theta},
	\end{equation}
	where  $\theta=3\big(\frac12-\frac1q\big).$
	
	If $u|_{\gamma}\ne0,$ then
	\begin{equation}\label{lady2}
		\|u\|_{L^q(\Om)}\le4^{\theta} C_{\Om}\|u\|^{\theta}_{H^1(\Om)}\|u\|^{1-\theta},
	\end{equation}
	where $C_{\Om}$ does not depend on a size of $\Om$.
\end{lemma}

\proof{See \cite{lady,lady2}.

\begin{lemma} \label{steklov} Let $v \in H^1_0(0,L).$ Then
	\begin{equation}\label{Estek} 
\|v_x\|^2(t)\geq \frac {\pi^2}{L^2}\|v\|^2(t).
	\end{equation}
\end{lemma}
	
	\begin{proof} The proof is based on the Steklov inequality: let $v(t)\in H^1_0(0,\pi)$, then $\int_0^{\pi}v_t^2(t)\,dt\geq\int_0^{\pi}v^2(t)\,dt.$
		Inequality \eqref{Estek} follows by a simple scaling.
	\end{proof}
	
\begin{lemma}\label{lemma3}
Let $f(t)$ be a continuous  positive function such that

\begin{align} & f'(t) + (\alpha - k f^n(t)) f(t) \leq 0.\label{lemao1}\\
& \alpha - k f^n(0)> 0.\label{lemao2}\end{align}
\noi Then

\begin{equation}f(t) < f(0)\end{equation}

\noi for all $t > 0$.
\end{lemma}

\proof{  Obviously, $f'(0) + (\alpha - k f^n(0)) f^n(0) \leq 0$. Since $f$ is continuous, there exists $T>0$ such that $f(t)<f(0)$ for every $t \in [0,T]$. Suppose that there is $\tau>0$ and $f(0) = f(\tau)$. Integrating \eqref{lemao1}, we find
	
	$$ \int_0^\tau (\alpha - k f^n(t)) f^n(t) \, dt < 0$$
	
	that contradicts \eqref{lemao2}. Therefore, $f(t) < f(0)$ for all $t > 0.$

(See also \cite{temam3}}).\\
The proof of Lemma 2.2 is complete.  $\Box$

\section{Existence of regular solutions}
\label{existence}
{ \bf Regularized problem.} To solve (\ref{1problema})-(\ref{3problema}), we exploit the parabolic regularization of this problem as follows:

For $\epsilon > 0$ (small), consider in $\Omega\times(0,t)$  the following parabolic problem:

	\begin{align}& L_\epsilon u_m^\epsilon = u_{mt}^\epsilon + u_{mx}^\epsilon + u_m^\epsilon u_{mx}^\epsilon + \Delta u_{mx}^\epsilon + \epsilon \partial^4 u_m^\epsilon = 0,\label{1para}\\
		&	u_m^\epsilon(0,y,z,t) =  u_m^\epsilon(L,y,z,t) =  u_{mx}^\epsilon(L,y,z,t) \nonumber\\ 
		&= \epsilon u_{mxx}^\epsilon(0,y,z,t) =  0 \label{2para},\\
		&	u_m^\epsilon(x,y,z,0) = u_{0m}(x,y,z),\\
		&  u_{0m}(0,y,z) = \epsilon u_{0mxx}(0,y,z)=u_{0m}(L,y,z)=u_{0mx}(L,y,z)= 0,\label{3para}
\end{align}
		
		\noi where
		
		$$\partial^4 u_m^\epsilon = \frac{\partial^4}{\partial x^4} u_m^\epsilon + \frac{\partial^4}{\partial y^4} u_m^\epsilon + \frac{\partial^4}{\partial z^4} u_m^\epsilon,$$
		
		\noi $u_0$ is an independent of $\epsilon$ approximation of $u_0$ such that for all $m \in \N$.\\
		
		Define 
\begin{align}& J_{0m}^\epsilon = ((1+x),|u_{0m^\epsilon} + u_{0mx}^\epsilon + u_{0m}^\epsilon u_{0mx}^\epsilon + \Delta u_{0mx}^\epsilon + \epsilon \partial^4 u_{0m}^\epsilon|^2)\\
	&J_{0m} = ((1+x), | u_{0m} + u_{0mx} + u_{0m}u_{0mx} + \Delta u_{0mx}|^2 ).	\end{align}

		It is known \cite{lady2, temam1, temam2} that there exists a unique regular solution of (\ref{1para})-(\ref{3para}), provided $u_{0m}$ is sufficiently smooth.
		
		\noindent Our goal is to obtain estimates for the $u_m^\epsilon$  independent of $m$ and $\epsilon$ with $u_{0m}$ sufficiently smooth, fixed; then to pass the limit as $\epsilon$ tends to $0$ getting a solution to (\ref{1problema})-(\ref{3problema}) with initial data $u_{0m}$. After that we pass to the limit as $m$ tends to $\infty$ and $u_{0m}$ tends to $u_0$, obtaining a solution to the original problem.
		
	We will assume that $u_{0m}$ converges to $u_0$ in the following sense:
	\begin{align*} J_{0m} \to J_0, \,\ ,\| u_{0myy} \| \to \| u_{0yy}\| \, \, \| u_{0myz} \| \to \| u_{0yz}\|,\,\, \| u_{0mzz} \| \to \| u_{0zz}\|, 
	\end{align*}
	\noi as $m \to \infty$. 	
		
		We assume that $m \geq m^*$, where $m^*$ is af natural number such that $J_{0m} \leq 2 J_0$. In turn, $\epsilon > 0$ is sufficiently small such that $J_{0m}^\epsilon \leq 4J_0$.
			
		\begin{lemma}\label{lem1}Under the conditions of Theorem \ref{main2}, for $m$ sufficiently large and $\epsilon$ sufficiently small, the following independent of $\epsilon$ and $m$ estimates hold:
			\begin{equation}\begin{cases}\label{est1}u_m^\epsilon \text{ is bounded in }L^{\infty}(0,T; L^2(\Om))\cap L^2(0,T; H^1(\Om)),\\
				u_{mx}^\epsilon \text{ is bounded in }L^2(0,T;L^2(\s)).\end{cases}\end{equation}
		\end{lemma}
		
			\proof
				{\bf Estimate I.}
				Multiply (\ref{1para}) by $u_m^\epsilon$ and integrate over $\Omega \times (0,t)$ to obtain
		
		\begin{equation*} \| u_m^\epsilon \|^2(t) + \int_0^t \int_\s (u_{mx}^\epsilon)^2(0,y,z,\tau)\, dy \, dz \, d\tau + 2\epsilon\int_0^t \| u_m^\epsilon \|_{[2]}^2(t) ds \leq \|u_{0m} \|^2
		\end{equation*}
		
		\noi and for $m$ sufficiently large
		
		\begin{equation}\label{last1e}\|u_m^\epsilon\|^2 (t) + \int_0^t \int_\s (u_{mx}^\epsilon)^2(0,y,z,\tau) \, d\tau \, ds  \leq 2 \|u_0\|^2.\end{equation}
		
		{\bf Estimate II}.
		Dropping the indices $m, \epsilon$, we transform the scalar product
				
		\begin{equation}2(L_\epsilon u_m^\epsilon, (1+x) u_m^\epsilon)(t) = 0 \end{equation}
		
		\noindent into the following equality:
		
		\begin{align}& \frac{d}{dt} ((1+x), u^2) (t) + (1-\epsilon)\int_\s u_x^2(0,y,z,t)\, dy \, dz - \|u\|^2(t)\nonumber \\ & + 3 \|u_x \|^2(t)+ \|u_y\|^2(t) + \|u_z\|^2(t) + 2\epsilon \mathcal{P}_1(t) = \frac{2}{3}(1,u^3)(t),\label{pree2}
		\end{align} 
		
		\noi where $\mathcal{P}_1 = ((1+x), u_{xx}^2 + u_{yy}^2 + u_{zz}^2)(t).$
		
Making use of \eqref{lady1}, we find
\begin{align}
	I=&\frac{2}{3}(1,u^3)(t)
	\le \frac{2}{3}\|u\|^3_{L^3(\Omega)}(t)\le\frac{2^4}{3}\big[\|\nabla u\|^{1/2}(t)\|u\|^{1/2}(t)\big]^3\nonumber\\
	&\le\frac{1}{2}\|\nabla u\|^2(t)+\frac{2^{11}}{3}\|u\|^6(t).\label{2pree2}
\end{align}
		
\noi By Lemma \ref{steklov},		
\begin{equation}\label{psteklov}\|u_x\|^2 (t) \geq \frac{\pi^2}{L^2} \|u\|^2(t).\end{equation}
	
\noi Substituting (\ref{last1e}), (\ref{2pree2}) and (\ref{psteklov}) into  \eqref{pree2}, we obtain for a fixed, sufficiently large $m$ that
		$$\frac{d}{dt} ((1+x), u^2)(t) + \big[ \frac{2\pi^2}{L^2} - 1 - \frac{2^{13}}{3} \|u_0\|^4 \big] \| u \|^2 (t) \leq 0.$$
		
\noi Under conditions of Theorem \ref{main2}, we have 

\begin{align}\label{3pree2}\frac{\pi^2}{L^2} - 1 - \frac{2^{13}}{3} \|u_0\|^4 \geq 0.\end{align}

Hence

$$\frac{d}{dt}((1+x),u^2)(t) + \frac{\pi^2}{L^2} ((1+x), u^2) (t) \leq 0$$

and

\begin{equation}\label{e2}\|u_m^\epsilon \|^2(t) \leq  2(1+L) \|u_0 \|^2 e^{-\chi t},\end{equation}
where $\chi = \frac{\pi^2}{L^2(1+L)}$. Returning to (\ref{pree2}), using (\ref{2pree2}) and (\ref{3pree2}), we obtain
\begin{align*}\int_0^t \big\{ \| \nabla u_m^\epsilon \|^2(\tau) + \int_\s u_{mx}^\epsilon(0,y,z,\tau)\, dy \, dz  \big\} \, d\tau \leq 2(1+L) \|u_0\|^2.\end{align*}

\begin{lemma}\label{lem2}Under the conditions of Theorem (\ref{main2}), for $m$ sufficiently large, the following  independent of $\epsilon$ and $m$ estimates hold:
			\begin{equation*}\begin{cases}\label{ests2} u_m^\epsilon \text{ is bounded in }L^{\infty}(0,T; H^1(\Om)),\\
					u_{mt}^\epsilon  \text{ is bounded in }L^\infty(0,T; L^2(\Om))\cap L^2(0,T; H^1(\Om)).
				\end{cases}\end{equation*}
\end{lemma}
\proof
		{\bf Estimate III}.	Dropping the indices $\epsilon$, $m$, transform the inner product 
		
		\begin{equation}\label{4pi}2(L_\epsilon u_m^\epsilon, (1+x) u_m^\epsilon)(t) = 0\end{equation}
		
		\noi into the equality	
		\begin{align}&\int_\s u_x^2(0,y,z,t)\, dy \, dz+ 3\|u_x\|^2(t) + \|u_y\|^2(t) + \|u_z\|^2(t)\nonumber\\
			&+ 2\epsilon \mathcal{P}_1 = \frac{2}{3} (1,u^3)(t) + \| u \|^2(t) - 2\big((1+x)u, u_t \big)(t).\label{nanana}
		\end{align}
		
		\noi By Holder and Young inequalities,
		\begin{align*}& 2\int_\Omega (1+x)uu_t \, d\Omega \leq 2 \|u\|(t) \big(\int_\Omega [(1+x) u_t]^2\, d\Omega \big)^{1/2}\\ 
			&\leq \|u\|^2(t) + \int_\Omega (1+x)^2u_t^2 \, d\Omega\\
			&\leq \|u\|^2(t) + (1+L)\big( (1+x), u_t^2\big)(t).
		\end{align*} 
		
		\noi Making use of \eqref{lady1}, \eqref{e2} and the last inequality, we reduce \eqref{nanana} to the form 
		\begin{align}\|\nabla u \|^2(t) &\leq 2\big[ 2 + \frac{2^{11}}{3} \|u \|^4 (t) \big]\|u\|^2(t) + 2((1+L)(1+x),(u_t)^2)(t)\nonumber\\
		&\label{e3}\leq 4(1+L)C_1 e^{-\chi t} + 2(1+L) (( 1+x), |u_{mt}^\epsilon|^2)(t),\end{align}
		
		\noi where $C_1 = 2+ \frac{2^{13}}{3} \|u_0 \|^4$ is independent of $t >0$.
		
Returning to (\ref{nanana}), we get
			\begin{align}&\| u_{mx}^\epsilon \|^2(t) \leq \frac{2}{5}(2 + \frac{2^{13}}{3} \|u_0\|^4) \|u_m^\epsilon\|^2(t)\nonumber \\
				& + \frac{2}{5} (1+L)\big((1+x), |u_{mt}^\epsilon|^2 \big)(t).\label{1e3} \end{align}
		
		{\bf Estimate IV}.		
		Consider the inner product
		\begin{equation}\label{ipe4}2((L_{\epsilon}u_m^\epsilon)_t, (1+x)u_{mt}^\epsilon)(t) = 0.\end{equation}
		
		\noi We calculate
		
		\begin{align}\label{ipe41}2 \big((1+x), u_t (uu_x)_t \big) &= \big( (1+x)u_x + u, u_t^2\big)\\
			&\leq \|(1+x)u_x + u\|\|u_t \|_{L^4(\Omega)}^2.\end{align}
		
		\noi Exploiting Lemma \ref{lemma1} and the Young inequality, we obtain
		
		\begin{align}&2((1+x),(uu_x)_t)(t) \leq \frac{1}{8} \| \nabla u_t \|^2(t) + (1+L)^4 3^3 2^{16} \big[ \| u_x \| ^4(t)\nonumber\\
		&+ \| u\|^4(t) \big]\|u_t\|^2(t)].\label{ipe42}
		\end{align}
		
		\noi Substituting (\ref{ipe41}) and (\ref{ipe42}) into \eqref{ipe4}, we get
		
		\begin{align}&\frac{d}{dt}((1+x),(u_t)^2)(t) + \big[\frac{23}{8}\| u_{xt}\|^2(t) + \frac{7}{8} \|u_{zt}\|^2(t)+ \frac{7}{8} \|u_{yt}\|^2(t) \big]\nonumber\\
		&\label{e41}-(1 + (1+L)^4 3^3 2^{16} \big[ \| u_x \| ^4(t) + \| u\|^4 \big]\|u_t\|^2(t) \leq 0.
		\end{align}
		
		Since by \eqref{1e3},
		
		$$\| u_x\|^4(t) \leq \frac{2^2}{25}C_1^2 \|u\|^4(t) + \frac{2^2}{25} (1+L)^2 ((1+x),u_t^2)^2(t)$$
		
		\noi and by Lemma \ref{steklov},
		$$\| u_{xt} \|^2(t) \geq \frac{\pi^2}{L^2} \|u_t\|^2(t),$$
		
		\noi then (\ref{e41}) reads
		
		\begin{align*}&\frac{d}{dt}((1+x),|u_t|^2)(t) + [\frac{\pi^2}{L^2} - 1 - 2^{16}3^3(1+L)^4(\frac{2^3}{25}C_1^2 + 1)\|u_0\|^4 \\ &- {2^{19}3^3}(1+L)^6(1+x,|u_t|^2)^2(t)  ]\|u_t\|^2(t)  \leq 0.\end{align*}
		According to Theorem \ref{main2} notations,	
		\begin{align}&\frac{d}{dt}((1+x),|u_{mt}^\epsilon|^2)(t) + \big[\frac{\pi^2}{L^2} - 1 - K_1 \|u_0\|^4  \nonumber \\& - K_2 ((1+x),|u_{mt}^\epsilon|^2)^2(t)\big] \|u_{mt}^\epsilon \|^2(t) \leq 0.\label{1e5} \end{align}

		\noi For $\epsilon$ small and $m$ sufficiently large fixed $2\epsilon^2 ((1+x), | \partial^4 u_{0m}^\epsilon|^2) \leq J_0$ and
		
		\begin{equation}\label{contaJ0} |J_{0m}^\epsilon|^2 \leq 5J_0^2
			\end{equation}

	\noi	By Lemma \ref{lemma3}, $((1+x), |u_{mt}^\epsilon|^2)(t) \leq J_{0m}^\epsilon$ for all $t > 0$ and making use of (\ref{1e5}), (\ref{contaJ0}), we obtain
	
\begin{align*}\frac{d}{dt} ((1+x),|u_{mt}^\epsilon|^2)(t) + \big[ \frac{\pi^2}{L^2} - 1 - K_1 \|u_0\|^4 - 25K_2 J_0^2 \big] \|u_t^\epsilon\|^2(t).
\end{align*}

\noi By \eqref{condition} and \eqref{contaJ0}, we get $\frac{\pi^2}{2L^2} - 1 - K_1 \|u_0\|^4 - 25K_2 J_0^2 \geq 0$.\\

Hence
	
$$\frac{d}{dt}((1+x), |u_{mt}^\epsilon|^2)(t) + \frac{\pi^2}{2L^2} ((1+x), | u_{mt}^\epsilon|^2)(t)\leq 0$$

\noi and

		\begin{equation}\label{decaiut} \|u_{mt}^\epsilon \|^2(t) \leq 2(1+L) J_{0} e^{-\frac \chi 2 t}.
		\end{equation}
		
	\noi	Returning to (\ref{e41}), we obtain
		
		\begin{align}\label{decaiutdps}\int_0^t \| \nabla u_{mt}^\epsilon\|^2(\tau) \, d\tau < C(J_{0},L)\end{align}
		
		\noi and from (\ref{e2}), (\ref{e3}),

		\begin{equation}\label{e4p} \| u_m^\epsilon \|_{H^1(\Omega)}^2(t) \leq C(L, J_{0}) e^{-\frac \chi 2 t}, \,\, \forall t>0.\end{equation}

	\noi for $\epsilon > 0$ sufficiently small and $m$ sufficiently large fixed.
	
		\begin{lemma}\label{prop1}Under the assumptions of Theorem \ref{main2}, we  find
		\begin{equation*}\begin{cases} u_{mxy}^\epsilon,u_{mxz}^\epsilon,u_{myy}^\epsilon, u_{mzz}^\epsilon, u_{myz}^\epsilon \text{ are bounded in }L^{\infty}(0,T; L^2(\Om)),\\
		
		u_m^\epsilon \text{ is bounded in }L^{2}(0,T; H^2(\Om)),\\
		
		\nabla u_{myy}, \nabla u_{mzz}, \nabla u_{mxyz} \text{ are bounded in } L^2(0,T; L^2(\Om)),\\
		
		u_{mx}^\epsilon(0,y,z,t) \text{ is bounded in }L^{\infty}(0,T; H^1(\s))\cap L^2(0,T; H^2(\s)).
			\end{cases}\end{equation*}
		\end{lemma}	
			
		\proof {\bf Estimate V}. Dropping the indices $\epsilon$, $m$, transform the scalar product 
		
		$$-2((1+x)Au_m^\epsilon,u_{myy}^\epsilon+u_{mzz}^\epsilon)(t)=0$$
		into the following equality:
		\begin{align}
			&\|u_y\|^2(t)+\|u_z\|^2(t)+(1-2\epsilon)\int_{\mathcal{S}}\big[u^2_{xy}(0,y,z,t)+u^2_{xz}(0,y,z,t)\big]\,dydz\notag\\
			&+\|u_{yy}\|^2(t)+\|u_{zz}\|^2(t)+2\|u_{yz}\|^2(t)+3\|u_{xy}\|^2(t)+3\|u_{xz}\|^2(t)\notag\\&
			+((1+x)u_x-u,u^2_y)(t)+((1+x)u_x-u,u^2_z)(t)\notag\\& + 2\epsilon \mathcal{P}_2(t)
			=2((1+x)u_t,u_{yy}+u_{zz})(t), \label{e5}
		\end{align}
		
		\noi where $\mathcal{P}_2 (t) = \big((1+x), |u_{xxy}^2 + u_{xxz}^2 + u_{yyy} + u_{yyz}^2 + u_{zzz}^2 + u_{zzy}^2|\big)(t)$.

		We estimate
		\begin{align*}
			&I_1=((1+x)u_x-u,u^2_y)(t)\leq \|(1+x)u_x-u\|(t)\|u_y\|^2_{L^4(\Om)}(t)\\&
			\leq (1+L)\big[\|u_x\|(t)+\|u\|(t)\big]4^{3/2}C^2_\Omega \|\nabla u\|^{1/2}(t)\|\nabla 
			u_y\|^{3/2}(t)\\&
			\leq \frac{1}{8}\|\nabla u_y\|^2(t)+(1+L)^4C^8_\Omega 2^{16¨}3^3\big[\|u_x\|^4(t)+\|u\|^4(t)\big]\|\nabla u\|^2(t).
		\end{align*}            
		Similarly,
		\begin{align*}
			& I_2=((1+x)u_x-u,u^2_z)(t)\leq  \frac18\|\nabla u_z\|^2(t)\\
			&+(1+L)^4C^8_\Omega 2^{16¨}3^3\big[\|u_x\|^4(t)+\|u\|^4(t)\big]\|\nabla u\|^2(t).
		\end{align*}
		Substituting $I_1,I_2$ into (\ref{e5}), we find		
		\begin{align*}&\int_{\mathcal{S}}\big[u^2_{xy}(0,y,z,t)+u^2_{xz}(0,y,z,t)\big]\,dy\, dz
			+\|u_{yy}\|^2(t)+\|u_{zz}\|^2(t)\\&+\|u_{yz}\|^2(t)+\|u_{xy}\|^2(t)+\|u_{xz}\|^2(t)\\&\leq C_4(L)\big[\|\nabla u\|^6(t)+\|u\|^4(t)\|\nabla u\|^2(t)+((1+x),u_t^2)(t)].
		\end{align*}
	
\noi Making use of (\ref{e4p}),
\begin{align}&\|u_{mxy}^\epsilon \|^2(t) + \|u_{mxz}^\epsilon \|^2(t) + \|u_{myy}^\epsilon \|^2(t) + \|u_{myz}^\epsilon \|^2(t) + \|u_{mzz}^\epsilon \|^2(t)\nonumber\\
&+\int_\s \{ (u_{mxy}^\epsilon)^2(0,y,z,t) + (u_{mxz}^\epsilon)^2(0,y,z,t)\} \, dy \, dz\nonumber\\
& \leq C(J_{0},L) e^{-\frac \chi 2 t}.\label{5e5}\end{align} 

To prove that $u$ is bounded in $L^\infty(0,T;H^2(\Omega))$, it is sufficiently to estimate $\| u_{mxx}^\epsilon \|(t)$.

	{\bf Estimate VI}. From the inner product
		\begin{equation}\label{e6inner} 2(L_\epsilon u_m^\epsilon, (1+x)\big[u_{myyyy}^\epsilon + u_{myyzz}^\epsilon + u_{mzzzz}^\epsilon\big])(t) = 0,\end{equation}
		
\noi dropping the indices $\epsilon$ and $m$, we find
		
	\begin{align*}&J_1 = 2\int_\Omega u_{t} (1+x)\big[ \partial_y^4 u + \partial_y^2 \partial_z^2 u + \partial_z^4 u \big]\, d\Omega\\ & = \frac{d}{dt} \big( (1+x), u_{yy}^2 + u_{yz}^2 + u_{zz}^2 \big)(t),\\
	&J_2=  2\int_\Omega u_{x} (1+x) \big[ \partial_y^4 u + \partial_y^2 \partial_z^2 u + \partial_z^4 u\big]\, d\Omega\\ & = \int_\Omega \big(u_{yy}^2 + u_{yz}^2 + u_{zz}^2 \big) \, d\Omega,\\
	&J_3 = 2\int_\Omega u_{xxx} (1+x) \big[ \partial_y^4 u + \partial_y^2 \partial_z^2 u + \partial_z^4 u\big] \, d\Omega = 3\big[\|u_{yyx} \|^2(t)\\
	& + \|u_{xyz} \|^2(t) + \|u_{zzx} \|^2(t)\big] + \int_\s \big\{u_{yyx}^2(0,y,z,t)+ u_{xyz}^2(0,y,z,t)\\ 
	&+ u_{zzx}^2(0,y,z,t)\big\}\, dy \, dz,\\
	& J_4 = 2\int_\Omega u_{yyx}(1+x)\big[\partial_y^4 u + \partial_y^2 \partial_z^2 u + \partial_z^4 u \big]\, d\Omega \\ 
	&= \|u_{yyy}\|^2(t) + \| u_{zzy} \|^2(t) + \|u_{yyz}\|^2(t),\\
	& J_5 = 2\int_\Omega u_{zzx}(1+x)\big[\partial_y^4 u + \partial_y^2 \partial_z^2 + \partial_z^4 u \big]\, d\Omega\\ 
	& = \|u_{zzz}\|^2(t) + \| u_{yyz} \|^2(t) + \|u_{zzy}\|^2(t),\\
	&J_6 = 2\epsilon \int_\Omega  \partial_y^4 u(1+x)\big[\partial_y^4 u + \partial_y^2 \partial_z^2 u + \partial_z^4 u \big]\, d\Omega\\ 
	& = \epsilon \big((1+x), u_{yyyy}^2 + u_{yyzz}^2 + u_{yyyz}^2\big)(t),\\
	& J_7 = 2\epsilon \int_\Omega \partial_z^4 u(1+x)\big[\partial_y^4 u + \partial_y^2 \partial_z^2 u + \partial_z^4 u  \big]\, d\Omega\\
	& = \epsilon \big((1+x), u_{zzzz}^2 + u_{yyzz}^2 + u_{zzzy}^2\big)(t),\\
	&J_8 = 2\epsilon \int_\Omega \partial_x^4 u (1+x)\big[\partial_y^4 u + \partial_y^2 \partial_z^2 u + \partial_z^4 u  \big]\, d\Omega = \epsilon ((1+x), u_{yyxx}^2 \\ 
	& + u_{xxyz}^2 + u_{zzxx}^2)(t) - 2\epsilon \int_\s \big\{u_{yyx}^2(0,y,z,t) + u_{xyz}^2(0,y,z,t) \\ & + u_{zzx}^2(0,y,z,t)\big\}\, dy \, dz.\end{align*}	

\noi Substituting $J_1$-$J_8$ into (\ref{e6inner}), we get
		
		\begin{align}
			&\frac{d}{dt}((1+x),u^2_{yy}+u^2_{zz}+u^2_{yz})(t) -\big[\|u_{yy}\|^2(t)+\|u_{zz}\|^2(t)+\|u_{yz}\|^2(t)\big]\notag\\&
			+(1-2\epsilon)\int_{\mathcal{S}}\big\{u^2_{xyy}(0,y,z,t)+u^2_{xzz}(0,y,z,t)+u^2_{xyz}(0,y,z,t)
			\big\}\,dydz\notag\\&
			+\|u_{yyy}\|^2(t)+\|u_{zzz}\|^2(t)+2\|u_{yzz}\|^2(t)+2\|u_{zyy}\|^2(t)			+3\|u_{xyy}\|^2(t)\notag\\ 
			& + 3\|u_{xzz}\|^2(t)+3\|u_{xyz}\|^2(t) + \epsilon \mathcal{P}_3 (t)
			+2((1+x)uu_x,\partial^4_y u)(t)\notag\\& +2((1+x)uu_x,\partial^4_z u)(t)+2((1+x)uu_x,\partial^2_y\partial^2_z u)(t)=0,\label{e7}
		\end{align}
		
		\noi where $\mathcal{P}_3 = \big((1+x), \big[u_{yyyy}^2 + 2u_{yyzz}^2 + u_{yyyz}^2 + u_{zzzz}^2 + u_{zzzy}^2 + u_{yyxx}^2 + u_{zzxx}^2\big]\big)(t)$. 
		
		We estimate the nonlinear term in the following manner:		
		\begin{align}& 2\int_\Omega uu_x (1+x) u_{yyyy} \, d\Omega = -2 \int_\Omega u_y u_x (1+x) u_{yyy} \, d\Omega \notag \\ & - 2\int_\Omega u u_{xy} (1+x) u_{yyy} \, d\Omega \notag \\
			&= 2\int_\Omega u_{yy}^2 u_x (1+x) \, d\Omega + 4 \int_\Omega u_y u_{xy} (1+x) u_{yy} \, d\Omega \notag\\ &+ 2\int_\Omega u u_{yyx} (1+x) u_{yy} \, d\Omega \notag\\
			&= 2\int_\Omega u_{yy}^2 u_x(1+x) \, d\Omega + 2\int_\Omega (u_y^2)_x (1+x) u_{yy} \, d\Omega \notag\\ &+  \int_\Omega u_{yy}^2 u_x (1+x) \, d\Omega \notag\\
			&=  \int_\Omega u_{yy}^2 u_x (1+x) \, d\Omega - 2\int_\Omega u_y^2 u_{yy} \, d\Omega - 2\int_\Omega u_y^2 (1+x) u_{yyx} \, d\Omega \notag\\ &- 2\frac{1}{2} \int_\Omega u u_{yy}^2 \, d\Omega \notag\\
			&=  \big( (1+x)u_x - u, u_{yy}^2  \big) -2\big( u_y^2, u_{yy} \big) - 2\big( (1+x) u_y^2 , u_{yyx} \big) \notag\\
			&= I_1 - I_2 - I_3.\label{Is1}
		\end{align}
		
		\noi Making use of (\ref{Is1}), we find
		
		\begin{align*} & I_1 = \big( (1+x)u_x - u, u_{yy}^2 \big)(t)\\
			&\leq  \|(1+x) u_x - u\|(t) \|u_{yy} \|_{L^4(\Omega)}^2(t)\\
			&\leq (1+L)[\|u_x\|(t) + \| u \|(t)]  4^{3/2} \|u_{yy} \|^{1/2}(t) \| \nabla u_{yy} \|^{3/2}(t)\\
			&\leq \frac{3}{4} \delta_1^{4/3} \| \nabla u_{yy} \|^2(t) + \frac{2^{16}}{\delta_1^4}\|u_{yy}\|^2(t) \big[\|u_x\|(t) + \| u \|\big]^4(t)\\
			&=\frac{3}{4} \delta_1^{4/3} \|\nabla u_{yy} \|^2(t) + \frac{2^{16}}{\delta_1^4} \|u_{yy}\|^2(t) \big[ \|u_x\|^4(t) + \|u\|^4(t) \big]\\
			&\leq \frac{3}{4}\delta_1^{4/3} \|\nabla u_{yy}\|^2(t) + C e^{-\frac \chi 2 t},\\		
	        & I_2 = -2(u_y^2, u_{yy})\\
	        &\leq 2\|u_y\|^4(t) + 2\|u_{yy}\|^2(t) \leq 2\| u_y \|_{L^4(\Omega)}^4(t) + 2\|u_{yy}\|^2(t)\\
			&\leq 4^4 \|u_y\| \|\nabla u \|^3(t) + 2\|u_{yy}\|^2(t)\\
			&\leq  2^8(\frac{1}{2} \|u_y\|^2(t) + \frac{1}{2} \| \nabla u \|^6(t)) + 2\|u_{yy}\|^2(t)\\			&\leq C e^{-\frac \chi 2 t},\\
			& I_3 = -2\big((1+x)u_y^2, u_{yyx} \big)(t) \\
			&\leq 2\|u_y\|_{L^4(\Omega)}^2(t) \| u_{yyx} \|(t)\\
			&\leq (1+L) 4^{3/2} \|u_y \|^{1/2}(t) \| \nabla u_y \|^{3/2}(t) \| u_{yyx} \|(t) \\
			&\leq \delta_2^2 \|u_{yyx} \|^2(t) + \frac{2^8}{\delta_2^2} \|u_y\|^2(t) \|\nabla u_y \|^3(t)\\
			&\leq \delta_2^2 \|u_{yyx} \|^2(t) + \frac{2^7}{\delta_2^2}\big( \|u_y\|^2(t) + \|\nabla u_y \|^6(t) \big)\\
			&\leq \delta_2^2 \|u_{yyx}\|^2(t) + C e^{-\frac \chi 2 t},
		\end{align*}

\noi where $\delta_1, \delta_2$ are arbitrary positive constants and $C$ is a constant independent of $\epsilon, m$ and $t$.	Substituting $I_1$-$I_3$ into (\ref{Is1}), we get
		
	\begin{align}&\int uu_{x} (1+x) u_{yyyy} \, d\Omega \leq \frac34{\delta_1}^{4/3} \|u_{yy}\|^2(t) + \delta_2^2 \|u_{yyx} \|^2(t)\nonumber\\& + C e^{-\frac \chi 2 t},\label{1e6}\end{align}
		
		\noi Similarly,				
				\begin{align}& \int uu_{x} (1+x) u_{zzzz} \, d\Omega \leq \frac34{\delta_1}^{4/3} \|u_{zz}\|^2(t) + \delta_2^2 \|u_{zzx} \|^2(t)\nonumber \\ & + C e^{-\frac \chi 2 t}\label{2e6}\end{align}	
				\noi and				
				\begin{align} & 2\int_\Omega uu_x(1+x) u_{yyzz}(t) \,d\Omega=-2((1+x)u_z u_x,\partial^2_y\partial_z u)(t)
				\notag\\ & -2((1+x)uu_{zx},\partial^2_y\partial_z u)(t) = 2((1+x)u_x,u^2_{zy})(t) \notag\\ &
				-((1+x)u^2_z,u_{xyy})(t)- (u^2_z,u_{yy})(t) + 2((1+x)uu_{xzz},u_{yy})(t)\\
				&\equiv J_1 + J_2 + J_3 + J_4.\label{Is2}
				\end{align}
				
				Using \eqref{e3}, \eqref{5e5}, we estimate 				
				\begin{align*}
				& J_1 = 2((1+x)u_x,u^2_{yz})(t) \leq 2(1+L)\|u_x\|(t)\|u_{yz}\|^2_{L^4(\Om)}(t)\\&
				\leq 2^4(1+L)\|u_x\|(t)\|u_{yz}\|^{1/2}(t)\|\nabla u_{yz}\|^{3/2}(t)\\&
				\leq \frac34 \delta_3^{4/3}\|\nabla u_{yz}\|^2(t)+\frac{2^{14}}{\delta_3^4}(1+L)^4\|u_x\|^4(t)\|u_{yz}\|^2(t)\\ 
				&\leq\frac34 \delta_3^{4/3}\|\nabla u_{yz}\|^2(t)+ C e^{-\frac \chi 2 t},\\
				& J_2=-((1+x)u^2_z,u_{xyy})(t) \leq (1+L)\|u_{xyy}\|(t)\|u_z\|^2_{L^4(\Om)}(t)\\&
				\leq \delta_4\|u_{xyy}\|^2(t)+\frac{4^3 (1+L)^2}{\delta_4}\|\nabla u\|(t)\|\nabla u_z\|^3(t)\\&
				\leq \delta_4\|u_{xyy}\|^2(t)+C e^{-\frac \chi 2 t},\\
				& J_3 = -( u^2_z,u_{yy})(t) \leq  \|u_{yy}\|(t)\|u_z\|^2_{L^4(\Om)}(t)\\&
				\leq \|u_{yy}\|^2(t)+4^2\|u_z\|(t)\|\nabla u_z\|^3(t)\\ &\leq \|u_{yy}\|(t)\|u_z\|^2_{L^4(\Om)}(t)
				\leq \|u_{yy}\|^2(t)+ C e^{-\frac \chi t 2},\\
				& J_4 = 2((1+x)uu_{xzz},u_{yy})(t)\\& \leq2(1+L)\|u_{xzz}\|(t)\|u\|_{L^4(\Om)}(t)\|u_{yy}\|_{L^4(\Om}(t)
				\leq \delta_5\|u_{xzz}\|^2(t) \\& +\frac{4^3(1+L)^2}{\delta_5}\|\nabla u_{yy}\|^{3/2}(t)\|u\|^{1/2}(t)\|\nabla u\|^{3/2}(t)\|u_{yy}\|^{1/2}(t)\\
				& \leq \delta_5\|u_{xzz}\|^2(t) \\& +\frac{3\delta_6^{4/3}}{4\delta_5}\|\nabla u_{yy}\|^2(t) +\frac{4^{11}(1+L)^8}{\delta_5 \delta_6^4}\|u\|^2(t)\|u_{yy}\|^2(t)\|\nabla  u\|^6(t)\\&
				\leq \delta_5\|u_{xzz}\|^2(t) +\frac{3\delta_6^{4/3}}{4\delta_5}\|\nabla u_{yy}\|^2(t)
				+C e^{-\frac \chi 2 t},
				\end{align*}
				
		\noi where $\delta_3, \delta_4$ are arbitrary positive constants and $C$ is a constant independent of $\epsilon, m$ and $t$. Substituting $J_1$-$J_4$ into (\ref{Is2}) and making use of (\ref{e7}), (\ref{Is1}), we reduce (\ref{e7}) to the form
		
		\begin{align*}
		&\frac{d}{dt}((1+x),u^2_{yy}+u^2_{zz}+u^2_{yz})(t)\\&
		+\int_{\mathcal{S}}\big\{u^2_{xyy}(0,y,z,t)+u^2_{xzz}(0,y,z,t)+u^2_{xyz}(0,y,z,t)
		\big\}\,dydz\\&
		+\|u_{yyy}\|^2(t)+\|u_{zzz}\|^2(t)+\|u_{yzz}\|^2(t)+\|u_{zyy}\|^2(t)\\&
		+\|u_{xyy}\|^2(t)+\|u_{xzz}\|^2(t)+\|u_{xyz}\|^2(t)\\&
		+\epsilon \big((1+x), \big[u_{yyyy}^2 + 2u_{yyzz}^2 + u_{yyyz}^2 + u_{zzzz}^2\\ & + u_{zzzy}^2 + u_{yyxx}^2 + u_{zzxx}^2\big]\big)(t)\notag
		\leq C(J_0,L)e^{-\frac \chi 2 t}.
		\end{align*}
		
		\noi Integrating over $(0,t)$ , we obtain
		
		\begin{align}
		&((1+x),(u^\epsilon_{myy})^2+(u^\epsilon_{mzz})^2+(u^\epsilon_{myz})^2)(t)\notag +\int_0^t\big\{\int_{\mathcal{S}}\big[(u^\epsilon_{mxyy})^2(0,y,z,\tau) \\ &+(u^\epsilon_{mxzz})^2(0,y,z,\tau)+(u^\epsilon_{mxyz})^2(0,y,z,\tau)
		\big]\,dydz\notag\\&
		+\|u_{myyy}^\epsilon\|^2(\tau)+\|u_{mzzz}^\epsilon\|^2(\tau)+\|u_{myzz}^\epsilon\|^2(\tau)+\|u_{mzyy}^\epsilon\|^2(\tau)\notag\\&
		+\|u_{mxyy}^\epsilon \|^2(\tau)+\|u_{mxzz}^\epsilon \|^2(\tau)+\|u_{mxyz}^\epsilon \|^2(\tau)\notag\\&
		+\epsilon \big((1+x), \big[(u_{myyyy}^\epsilon)^2 + 2(u_{myyzz}^\epsilon)^2 + (u_{myyyz}^\epsilon)^2 + (u_{mzzzz}^\epsilon)^2\notag\\& + (u_{mzzzy}^\epsilon)^2 + (u_{myyxx}^\epsilon)^2 + (u_{mzzxx}^\epsilon)^2 \big]\big)\big\}(\tau)\,d\tau\notag\\&
		\leq C(L,J_0)((1+x),u^2_{0yy}+u^2_{0zz}+u^2_{0zy})\label{E7}\end{align}
		
		\noi for $\epsilon$ sufficiently small and $m$ fixed and sufficiently large, with the constant $C(L,J_0)$ independent of $t>0.$

		{\bf Estimate VII}. Dropping the indices $\epsilon, m$ and the variables $y,z,t$, rewrite (\ref{1para})-(\ref{3para}) in the form		
		\begin{align}& u_{xxx} + u u_x + \epsilon u_{xxxx} = g, \label{1newpara}\\
	&			u(0)= u_{xx}(0) = u(1) = u_x(1)  = 0.\label{2newpara}
			\end{align}

		\noi By (\ref{E7}), $g \in L^2(0,T;L^2(\Om))$. Multiplyng (\ref{1newpara}) by $x$ and integrating over $x \in (0,1)$, we get

		\begin{equation}\label{e71}u_x(0) + u_{xx}(1) - \frac{1}{2} \int_0^1 u^2 \, dx - \epsilon u_{xx}(1) + \epsilon u_{xxx}(1) = \int_0^1 x g\,  dx.\end{equation}
		
		\noi Integrating \eqref{1newpara} over $(x,1)$ gives
		
		\begin{equation}\label{e72}u_{xx}(1) - u_{xx}(x) - \frac{1}{2} u^2 (x) + \epsilon u_{xxx}(1) - \epsilon u_{xxx} (x) = \int_x^1 g \, dx.\end{equation}
		
		\noi Subtracting (\ref{e72}) from (\ref{e71}), we obtain
		$$u_x(0) - \frac{1}{2} \int_0^1 u^2 \, dx - \epsilon u_{xx}(1) + u_{xx} + \frac{1}{2} u^2 + \epsilon u_{xxx} = \int_0^1 gx \, dx - \int_x^1 g \, d\xi.$$ 
		
Define
		
		$$h(x) = -u_x (0) + \frac{1}{2} \int_0^1 u^2 \, dx - \frac{1}{2} u^2 + \int_0^1 gx \, dx - \int_x^1 g \, d\xi.$$
		
		\noi Then (\ref{1newpara}) reads
		
		\begin{equation}\label{3newpara}u_{xx} + \epsilon u_{xxx} = \epsilon u_{xx}(1) + h.\end{equation}
		
		\noi Multiplying (\ref{3newpara}) by $u_{xx}$ and integrating over $(0,1)$, we find
		
		\begin{align*}&\epsilon \int_0^1 u_{xxx} u_{xx} \, dx = \frac{\epsilon}{2} u_{xx}^2(1),\\
		&\epsilon \int_0^1 u_{xx} (1) u_{xx} \, dx = \epsilon u_{xx}(1) \int_0^1 u_{xx} \, dx = -\epsilon u_{xx}(1) u_{x}(0).
		\end{align*}
		
		\noi Hence, (\ref{3newpara}) becomes
		
		\begin{align*} & \int_0^1 u_{xx}^2 \, dx + \frac{\epsilon}{2} u_{xx}^2 (1) = - \epsilon u_{xx}(1) u_{x} (0) + \int_0^1 u_{xx} h\, dx\\
			&\leq \frac{\epsilon}{4} u_{xx}^2 (1) + 4 \epsilon u_x^2(0) + \frac{1}{2} u_{xx}^2 + \frac{1}{2} \int_0^1 h^2 \, dx.
		\end{align*}
		
	\noi	Therefore

		\begin{equation} \int_0^1 u_{xx}^2 \, dx + \frac{\epsilon}{4} u_{xx}(1) \leq 4\epsilon u_{x}^2(0) + \frac{1}{2} \int_0^1 h^2 \, dx\end{equation}
		
	\noi 	and 
		
		$$\|u_{mxx}^\epsilon \|^2 (t) \leq C(J_0, L, \|u_{0yy}\|, \|u_{0yz}\|, \|u_{0zz}\|).$$
{\bf Passage to the limit as $\epsilon \to 0$.}	Using the estimates obtained in Lemmas \ref{lem1}, \ref{lem2} \ref{prop1} and compactness arguments, we can pass to the limit as $\epsilon \to 0$ in (\ref{1para}-\ref{3para}) and get a solution $u_m$ for (\ref{1problema})-(\ref{3problema}) with initial data $u_{0m}$ for $m$ large and fixed:
	\begin{align}
		&   u_{mt} + (1 + u_m)u_{mx} + u_m u_{mx} + \Delta u_{mx} = 0;  \label{1problemao}\\
		&	u_m(0,y,z,t) = u_m(L,y,z,t) =  u_{mx}(L,y,z,t) = 0, \label{2problemao}\\
		&	u_m(x,y,z,0) = u_{0m}(x,y,z),\,\,\,\,  (x,y,z) \in \Omega \label{3problemao}
	\end{align}
	
\noi such that
		
	\begin{align}&\label{1class}u_m \in L^\infty(0,T; H^1(\Omega))\cap L^2(0,T; H^2(\Omega)),\\
&	\label{2class}u_{mt} \in L^\infty(0,T; L^2(\Omega)) \cap L^2(0,T; H^1(\Omega)) \text{ for all } t > 0.\end{align}
		
		Rewriting \eqref{1problemao} as 
		
		$$u_{mxxx} = -u_{mt} -  u_{myyx} - u_{mzzx} - u_{mx} - u_m u_{mx},$$
		
	\noi and making use of \eqref{decaiut}, \eqref{decaiutdps}, \eqref{E7}, we find that
		
		$$u_{mxxx} \in L^2(0,T; L^2(\Omega)), \text{ for all } T>0.$$
		
{\bf Passage to the limit as $m \to \infty$.} Since the constants of estimates in Lemmas \ref{lem1}, \ref{lem2} \ref{prop1} do not depend on $\epsilon, m, t$, we can pass to the limit in (\ref{1problemao})-(\ref{3problemao}) as $m \to \infty$ and obtain a solution $u(x,y,z,t)$ for (\ref{1problema})-(\ref{3problema}) such that
		\begin{align}\label{1classu}&u \in L^\infty(0,\infty; H^1(\Omega))\cap L^2(0,\infty; H^2(\Omega)),\\
		&u(0,y,z,t) \in L^\infty(0,T;H^1(\s))\cap L^2(0,T; H^2(\s)),
			\end{align}
	
		\begin{align}&\| u \|_{H^1(\Omega)}^2(t) + \| u_y \|_{H^1(\Omega)}^2(t) + \| u_z \|_{H^1(\Omega)}^2(t) + \int_\s \{u_{xy}^2(0,y,z,t)\nonumber \\ & + u_{xz}^2(0,y,z,t)\} \, dy \, dz \leq C(J_0, L) e^{-\frac{\chi}{2} t}, \text{ for all } t > 0.\label{decaiproboriginal1}\end{align}
		
		{\bf Regularity of u}. Making use of (\ref{e7}) and (\ref{decaiproboriginal1}), write (\ref{1problema})-(\ref{3problema}) in the form
		\begin{align}&\Delta u_{x} = - u_{t} - u_{x} - \frac{1}{2}(u^2)_x,\label{1elip}\\
			&	u_{x}(L,y,z,t) = 0,\label{2elip}\\
			&	u_{x}(0,y,z,t) = \phi(y,z,t) \in L^\infty(0,T; H^1(\s))\cap L^2(0,T;H^2(\s)).\label{3elip}
			\end{align}
			
			\noi Denoting $v = u_x - \phi(y,z,t) (L-x)$, we get
			\begin{equation}\Delta v = -u_t - v - \frac{1}{2} (u^2)_x + \Delta (\phi(y,z,t)(L-x)) \equiv F(x,y,z,t),
			\end{equation}
			
			\noi where $F \in L^\infty(0,T, H^{-1} (\Omega))$.
			
\noi		 From the inner product
			
			$$(\Delta v, v ) = (F,v),$$
			
\noi			we calculate

			$$\int_\Omega (\phi_y(x-L))_y (u_x - \phi(x-L)) \, d\Omega \leq  \frac{3}{2} \int_\s \phi_y^2 \, dy \, dz + \frac{1}{2} \| u_{xy} \|^2(t),$$
			$$\int_\Omega (\phi_z(x-L))_z (u_x - \phi(x-L)) \, d\Omega \leq  \frac{3}{2} \int_\s \phi_z^2 \, dy \, dz + \frac{1}{2} \| u_{xz} \|^2(t)$$
			
\noi and come to the inequality
			\begin{align*}&\|v_x\|^2(t) + \|v_y\|^2(t) + \|v_z\|^2(t) \leq \|u_t \|^2(t) + \frac{1}{2} \|u_x\|^2(t)\\ &
				 + \frac{1}{2}\int_\s u_x^2 (0,y,z,t) \, dy \, dz  + \frac{3}{2} \|u_x\|^2(t) + \frac{1}{2}  \int_\s u_x^2 (0,y,z,t) \, dy \, dz \\ &+ \frac{1}{4} \|u\|_{L^4(\Omega)}^4(t) + \frac{1}{4} \|v_{x}\|^2(t) +  \frac{1}{2} \|u_x\|^2(t) + \int_\s u_x^2(0,y,z,t) \, dy \, dz\\
				&+ \frac{3}{2}\int_\s u_{xy}^2(y,z,t) + u_{xz}^2(y,z,t) \, dy + \frac{1}{2} \|u_{xy} \|^2(t) + \frac{1}{2} \|u_{xz}\|^2(t).
			\end{align*}
			
			\noi Since $ \|u\|_{L^4}^4(t) \leq 4^3 \|u \|(t) \|\nabla u \|^3(t)$, making use of (\ref{decaiproboriginal1}), we get
			
			$$\|u_{xx}\|^2(t) \leq C(J_0, L) e^{-\frac{\chi}{2} t} \text{ for all } t>0.$$
			
			Returning to (\ref{1elip})-(\ref{3elip}), we can see that $F \in L^2(0,\infty; L^2(\Omega))$. This implies that $u \in L^2(0,\infty; H^2(\Omega))$. Hence $u \in L^\infty(0,\infty;H^2(\Omega))\cap L^2(0,\infty; H^3(\Omega))$.
			
			{\bf Regularity of $\Delta u_x$}. Writing $\Delta u_x = -u_t - u_x - u u_x$, and recalling that	
			$$u_x, u_t, u u_x \in L^\infty (0,T; L^2(\Omega)),$$	
			\noi we estimate	
			\begin{align}
			&	\| u u_{x} \|_{H^1(\Omega)}(t) \leq \|  u u_x \|(t) + \| \nabla (u u_x) \|(t)\notag \\
				&\leq \| u \|_{L^4(\Omega)} (t) \| u_x \|_{L^4(\Omega)} (t) + \| u \|_{L^4(\Omega)} (t) \| \nabla u_x \|_{L^4(\Omega)} (t) + \| \nabla u \|_{L^4(\Omega)}^2 (t)\notag \\
				&\leq 2^3 [ \|u \|^{1/4} (t) \| \Delta u \| (t) \| \nabla u_x \|^{3/4}(t)\notag \\
				&+ \|u \|^{1/4} (t) \| \nabla u \|^{ 3/4} (t) \| \nabla u_{x} \|^{1/4} (t) \| \Delta u_x \|^{3/4} (t) \|  \Delta u \|^{3/2}(t) ]\notag\\
				& + \| \nabla u \|^{1/4}(t) \| \Delta u \|^{3/4}(t) \leq C \| u \|_{H^3(\Omega)}^2(t).\label{h1norm}
			\end{align}
			
			\noi where C is a constant independent of $t$. By (\ref{1class}), (\ref{2class}) and (\ref{h1norm}) read
			
			$$u_t, u_x, uu_x \in L^2(0,T; H^1(\Omega)).$$
			
			\noi Hence
			
			$$\Delta u_x \in L^\infty (0,\infty; L^2(\Omega)) \cap L^2(0,\infty; H^1(\Omega)).$$
			
			\noi This proves the existence part of Theorem \ref{main2}.
\section{Uniqueness of a regular solution and continuous dependence on initial data}
\label{uniqueness}

			\begin{thm}A global regular solution to (\ref{1problema})-(\ref{3problema}) is uniquelly defined.\end{thm}
			
			\proof 
			Let $u_1, u_2$ be two distinct solutions to (\ref{1problema})-(\ref{3problema}) and $w = u_1 - u_2$. Then
			
			\begin{align*}& w w_x + (wu_2)_x = (u_1 - u_2) (u_1 - u_2)_x + (u_1 - u_2)_x u_2 + (u_1 - u_2) u_{2x}\\
				&= u_1 u_{1x} - u_1{u_2x} - u_2 u_{1x} + u_2 u_{2x} + u_{1x} u_2 - u_{2x} u_2 + u_{1} u_{2x} - u_2 u_{2x}\\
				&= u_1 u_{1x} - u_2 u_{2x}
			\end{align*}
			
			\noi and (\ref{1problema})-(\ref{3problema}) can be rewritten in the form
			
			\begin{align}& Lw = w_t + w_x + \Delta w_x + w w_x+ (w u_2)_x,\\
				&	w(0,y,z,t) = w(L,y,z,t) = w(L,y,z,t) = 0,\\
				&	w(x,y,z,0) = 0.
			\end{align}

			\noi Transform the inner product
			
			$$2((1+x)Lw, w)(t) = 0$$ 
			
			\noi into the following equality
			
			\begin{align*}&\frac{d}{dt} ((1+x), w^2)(t) + \int_\s w_{x}^2(0,y,z,t)\, dy \, dz + 2\| w_x \|^2(t) + \| \Delta w \|^2(t)\\ &- \frac{2}{3} (1,w^3)(t) + ((1+x)u_{2x} - u_2, w^2)(t)  = 0.\end{align*}
			
			Using Lemma \ref{lemma1}, we find
			
			\begin{align*}
				&I_1=-\frac23(1,w^3)(t)\leq \frac34 \delta^{4/3}\|\nabla w\|^2(t)+\frac{2^{14}}{3^4\delta^4}\|w\|^6(t)
			\end{align*}
			
			\noi and
			
			\begin{align*} & I_2 = ((1+x)u_{2x} - u_2, w^2)(t) \leq \| (1+x) u_{2x} - u_2 \|(t) \| w \|_{L^4}^2(t)\\
				&\leq (1+L)[\|u_{2x}\|(t) + \|u_2 \|(t)] 4^{3/2} \|w \|^{1/2}(t) \| \nabla w \|^{3/2}(t)\\
				&\leq \frac{C}{ \delta ^4} (\|u_{2x} \|^4(t) + \| u_2 \|^4(t) \|w \|^2(t)) + \frac{3}{4}\delta^{4/3} \| \nabla w \|^2(t).
			\end{align*}
			
		  \noi For $\delta >0 $ sufficiently small, we find
			
			$$\frac{d}{dt} (1+x, w^2)(t) \leq C \big[ \| u_{2x} \|^4(t) + \| u_1 \|^4(t) + \|u_2 \|^4(t) \big] (1+x, w^2)(t).$$
			
\noi By the Grownwall Lemma,
			$$\|w \|^2(t) \leq ((1+x), w^2)(t) \equiv 0.$$
			
			\begin{rem}
				If $w(x,y,z,0)=w _0(x,y,z)\ne 0,$ then
				$$\|w\|^2(t)\leq ((1+x),w^2)(t)\leq C(L,J_0)((1+x),w^2_0)\;\forall t>0.$$
				This means continuous dependence of solutions to (\ref{1problema})-(\ref{3problema}) on initial data.
			\end{rem}

	\begin{rem}
		The geometrical restriction  $L \leq \frac{\pi}{2}$ in Theorem 2.1 is caused by the presence of the term $u_x$ in (2.1) and is connected with spectral properties of the linear spatial operator  $u_x+\Delta u_x$ and existing of critical size domains (see \cite{larh1} in 2D case). On the other hand, there are some boundary conditions under which there are not critical size domains \cite{doronin3}. We need also small initial data in order to suppress destabilizing effects of the nonlinear  convective term $uu_x$. We must note that in \cite{lar3} such restrictions for $c_su_x$ and initial data did not appear while establishing the existence of weak solutions for the 3D ZK equation, but there was an open problem, still unresolved, on uniqiueness of this weak solution. 	\end{rem}

	{\bf Conclusions.} We have established the existence and uniqueness of global regular solutions to (\ref{1problema})-(\ref{3problema}) as well as exponential decay of
	 the $H^2$-norm exploiting an approach of proving simultaneously existence  and exponential decay. Therefore geometrical restrictions and ``smallness" conditions for initial data have appeared. Of course, theses restrictons are not necessary while proving only existence and uniqueness of global regular solutions for the 2D  ZK equation. Nevertheless. similar restrictions appear while proving exponential decay of the existing global regular solutions \cite{larh1}.

\
\medskip

\bibliographystyle{torresmo}

\end{document}